\documentclass[a4paper]{article}
\usepackage{graphicx}
\usepackage{hyperref}
\hypersetup{
    colorlinks,
    citecolor=black,
    filecolor=black,
    linkcolor=black,
    urlcolor=black
}

\usepackage{amssymb}
\usepackage{amsmath}
\usepackage{amsfonts}
\usepackage{url}
\usepackage{amsthm}
\usepackage{enumitem}
\usepackage{tikz-cd}
\usetikzlibrary{shapes.geometric}
\usetikzlibrary{arrows}
\usepackage{authblk}
\usepackage{booktabs}

\newtheorem{theorem}{Theorem}
\theoremstyle{definition}
\newtheorem*{remark}{Remark}
\newtheorem{problem}{Problem}

\newcommand{\m}{\mathbf} 
\newcommand{\rd}{/}
\newcommand{\ld}{\backslash}
\newcommand{\jn}{\mathbin{\vee}}
\newcommand{\mt}{\mathbin{\wedge}}
\newcommand{\f}{\varphi}
\newcommand{\ps}{\psi}

\title{Residuated lattices do not have the amalgamation property}
\author[1]{Peter Jipsen}
\author[2]{Simon Santschi}

\affil[1]{Faculty of Mathematics, Chapman University, Orange, CA, USA}
\affil[2]{Mathematical Institute, University of Bern, Bern, Switzerland}
\date{\vspace{-5ex}}
\begin{document}
\maketitle
\begin{abstract}
We show that the variety of residuated lattices does not have the amalgamation property, thereby settling a long-standing open problem. In addition, we show that the amalgamation property fails for several subvarieties, including idempotent residuated lattices, involutive residuated lattices, and (integral)  distributive residuated lattices.
\end{abstract}
A \emph{residuated lattice} is an algebraic structure $\langle R,\mt, \jn, \cdot,\ld,\rd,1\rangle$ such that 
    $\langle R,\cdot,1 \rangle$ is a monoid, 
    $\langle R, \mt, \jn \rangle$ is a lattice,
    and for all $a,b,c\in R$
    \[
    a\cdot b \leq c \iff b\leq a\ld c \iff a \leq c \rd b.
    \]
Residuated lattices form a variety (or equivalently an equational class), i.e., the above equivalence can be expressed using equations only.
A \emph{pointed residuated lattice} is an algebraic structure $\langle R,\mt, \jn, \cdot,\ld,\rd,1,0\rangle$ such that the reduct $\langle R,\mt, \jn, \cdot,\ld,\rd,1\rangle$ is a residuated lattice and $0$ is a constant. An \emph{involutive} residuated lattice is a pointed residuated lattice that satisfies the equation $0\rd (x\ld 0) \approx  (0\rd x)\ld 0 \approx x$; it is called \emph{cyclic} if it satisfies $0\rd x \approx x\ld 0$ and it is called \emph{odd} if $1=0$. 
An element $a$ of a residuated lattice is called \emph{central} if it commutes with every other element and \emph{idempotent} if $a\cdot a = a$. 

A  residuated lattice is called \emph{commutative/idempotent} if its monoid reduct is commutative/idempotent, \emph{integral} if it satisfies $x\leq 1$, and \emph{distributive} if its lattice reduct is distributive. 
Pointed residuated lattices provide the algebraic semantics for substructural logics, that is, the axiomatic extensions of the Full Lambek Calculus (FL).
For more details about residuated lattices and substructural logic we refer to \cite{GJKO07,MPT23}. 

A class $\mathsf{K}$ of algebraic structures is said to have the \emph{amalgamation property} (AP) if for each pair of embeddings $\f_B \colon \m{A} \to \m{B}$, $\f_C \colon \m{A} \to \m{C}$, with $\m{A},\m{B},\m{C} \in \mathsf{K}$, there exists an algebraic structure $\m{D} \in \mathsf{K}$ and embeddings $\ps_B \colon \m{B} \to \m{D}$ and $\ps_C \colon \m{C} \to \m{D}$, such that $\ps_B \circ \f_B = \ps_C\circ \f_C$, i.e., the following diagram commutes
\[\begin{tikzcd}
	& {\bf B} \\
	{\bf A} && {\bf D} \\
	& {\bf C}
	\arrow["{\f_B}", hook, from=2-1, to=1-2]
	\arrow["{\ps_B}", dashed, hook, from=1-2, to=2-3]
	\arrow["{\f_C}"', hook, from=2-1, to=3-2]
	\arrow["{\ps_C}"', dashed, hook, from=3-2, to=2-3]
\end{tikzcd}\]
The pair $\langle \f_B,\f_C\rangle$ is called a \emph{span} and the pair $\langle \ps_B,\ps_C \rangle$ is called an \emph{amalgam} of the span $\langle \f_B,\f_C\rangle$ in $\mathsf{K}$. If, furthermore, $(\ps_B\circ\f_B)[A] = \ps_B[B] \cap \ps_C[C]$, then $\langle \ps_B,\ps_C \rangle$ is called a \emph{strong amalgam}; and if every span in $\mathsf{K}$ can be completed with a strong amalgam we say that $\mathsf{K}$ has the \emph{strong amalgamation property}.

It was a long-standing open problem whether the variety of residuated lattices has the amalgamation property; see, e.g., \cite{Takamura04} and \cite{MPT10}. Even though there has been a lot of progress recently in establishing the AP or the failure thereof for various subvarieties of residuated lattices (see, e.g., \cite{FMS2023,FS2024}) the most general case remained open. We close this gap, establishing the failure of the AP for the variety of residuated lattices. Additionally, we resolve the same question for various other varieties, e.g., the varieties of idempotent and involutive residuated lattices.

\begin{table}
\centering
\begin{tabular}{lcc}
&Central & Non-central \\
Idempotent & 
\begin{tikzpicture}
\draw[draw = black,fill= black, minimum size = 6pt, inner sep = 0pt] circle(3pt);
\end{tikzpicture}
 &
\tikz[]\fill[black] (0,0) rectangle (0.225,0.225);
  \\
Non-idempotent 
&
\begin{tikzpicture}
\draw[draw = black,fill= white, minimum size = 6pt, inner sep = 0pt] circle(3pt);
\end{tikzpicture}
 & 
 \tikz[]\draw[black] (0,0) rectangle (0.225,0.225);\\
 & & \\
\end{tabular}
\caption{Convention for nodes in Hasse diagrams}
\label{tab:convention-hasse}
\end{table}

In what follows we will show the failure of the AP for the different varieties by considering spans of finite residuated lattices for which we will give Hasse diagrams. As the residuals are uniquely determined by the product and the order, the labels of the Hasse diagrams will only contain enough information about the product to uniquely determine it modulo the axioms of residuated lattices. To simplify the presentation we will use the following convention (see Table~\ref{tab:convention-hasse}): central elements are labeled with round nodes and non-central elements with square nodes; a filled node indicates that an element is idempotent.

\begin{figure}
\centering
\begin{tikzpicture}[
fbul/.style={circle,draw=black,fill=black, minimum size = 4pt, inner sep = 0pt},
fcirc/.style={circle,draw=black, minimum size = 4pt, inner sep = 0pt},
square/.style={regular polygon,regular polygon sides=4},
fsqu/.style={square,draw=black,fill=black, minimum size = 5.5pt, inner sep = 0pt}
]
  \node[fbul] (top) at (0,-2.4) {};
  \node[fbul] (1) at (0,-3.2) {};
  \node[fbul] (bot) at (0,-4) {};
   
   \node[right] () at (top) {$\top$};
   \node[right] () at (1) {$1$};
   \node[right] () at (bot) {$\bot$};

   \node[below] (A) at (0,-4.5) {$\m A$};

  \draw  (top) -- (1) -- (bot);
\end{tikzpicture}
\qquad
\begin{tikzpicture}[
fbul/.style={circle,draw=black,fill=black, minimum size = 4pt, inner sep = 0pt},
fcirc/.style={circle,draw=black, minimum size = 4pt, inner sep = 0pt},
square/.style={regular polygon,regular polygon sides=4},
fsqu/.style={square,draw=black,fill=black, minimum size = 5.5pt, inner sep = 0pt}
]
   \node[fsqu] (top) at (0,-2.4) {};
   \node[fbul] (a) at (0,-3.2) {};
   \node[fsqu] (b) at (0.7,-4) {};
   \node[fbul] (1) at (-0.7,-4) {};
   \node[fbul] (bot) at (0,-4.8) {};
   
   \node[right] () at (top) {$\top = \top b$};
   \node[right] () at (a) {$a$};
   \node[right] () at (b) {$\arraycolsep1.3pt\begin{array}{ll} b &= b\top \\
   &= ba \end{array}$};
   \node[left] () at (1) {$1$};
   \node[right] () at (bot) {$\bot$};

   \node[below] (B) at (0,-5.3) {$\m B$};
  
  \draw  (top) -- (a) -- (1) -- (bot) -- (b) -- (a);
\end{tikzpicture}
\begin{tikzpicture}[
fbul/.style={circle,draw=black,fill=black, minimum size = 4pt, inner sep = 0pt},
fcirc/.style={circle,draw=black, minimum size = 4pt, inner sep = 0pt},
square/.style={regular polygon,regular polygon sides=4},
fsqu/.style={square,draw=black,fill=black, minimum size = 5.5pt, inner sep = 0pt}
]
  \node[fbul] (top) at (0,-3.2) {};
  \node[fbul] (c) at (-0.7,-4) {};
   \node[fbul] (d) at (0.7,-4) {};
  \node[fbul] (1) at (0,-4) {};
  \node[fbul] (bot) at (0,-4.8) {};
   
   \node[right] () at (top) {$\top$};
   \node[right] () at (d) {$d = d\top$};
   \node[left] () at (c) {$c = c\top$};
   \node[left] () at (1) {$1$};
   \node[right] () at (bot) {$\bot=cd$};
   
   \node[below] (C) at (0,-5.3) {$\m C$};

  \draw  (top) -- (c) -- (bot) -- (d) -- (top) -- (1) -- (bot);
\end{tikzpicture}

\caption{
Counterexample for the AP in (idempotent) residuated lattices (Theorem~\ref{thm}).
}
\label{fig}
\end{figure}

\begin{theorem}\label{thm}
    The varieties of residuated lattices and idempotent residuated lattices do not have the amalgamation property. 
\end{theorem}

\begin{proof}
Consider the residuated lattices $\m{A}$, $\m{B}$ and $\m{C}$ depicted in Figure~\ref{fig}. 
It is straightforward to check that the inclusion maps $\m{A} \hookrightarrow \m{B}$ and $\m{A}\hookrightarrow \m{C}$ are embeddings, i.e., they form a span in the variety of residuated lattices. So assume for a contradiction that the span has an amalgam in the variety of residuated lattices. Then, since $1 \jn a = 1 \jn b = a < \top$ in $\m{B}$ and  $1\jn c = 1\jn d = \top$ in $\m{C}$, we may assume that $A = B \cap C$ and there exists a residuated lattice $\m{D}$ such that $\m{B}\leq \m{D}$ and $\m{C} \leq \m{D}$.
First, since $c = c\top$ and $\top b = \top$, in $\m{D}$ we have 
\begin{equation}\label{eq1}
    cb = c\top b = c\top = c. \tag{$\ast$}
\end{equation}
Moreover, using that $\top = 1 \jn c$ and $cd = \bot$, we get
\begin{equation*}
    d = \top d = \top b d = (1 \jn c)bd = bd \jn cbd \overset{\eqref{eq1}}{=} bd \jn cd = bd \jn \bot = bd,
\end{equation*}
where for the last equality we use that $\bot \leq d$ implies $\bot = b\bot \leq bd$. But then also 
\begin{equation*}
    b = b\top = b(1\jn d) = b\jn bd = b \jn d,
\end{equation*}
yielding $d \leq b \leq a$. Hence, since $1\leq a$, we obtain
\begin{equation*}
    \top = 1 \jn d \leq a \jn d = a,
\end{equation*}
contradicting the fact that $a < \top$. Therefore the variety of residuated lattices does not have the amalgamation property. In particular, the residuated lattices $\m{A}$, $\m{B}$, and $\m{C}$ are idempotent, hence also the variety of idempotent residuated lattices does not have the amalgamation property.
\end{proof}

 It is shown in \cite[Corollary 5.2]{GJS2025} that every cyclic involutive idempotent residuated lattice is commutative. Since commutative involutive residuated lattices have the AP \cite{GJKO07}, we need an additional counterexample for showing the next theorem.

\begin{theorem}\label{thm2}
    The variety of (odd) (cyclic) involutive residuated lattices does not have the amalgamation property.
\end{theorem}
\begin{proof}
Consider the algebras $\m{A}$, $\m{B}$, and $\m{C}$ depicted in Figure~\ref{fig2}. They are odd cyclic involutive residuated lattices. It is straightforward to check that the inclusion maps $\m{A} \hookrightarrow \m{B}$ and $\m{A} \hookrightarrow \m{C}$ are embeddings, so they form a span in the variety of involutive residuated lattices. Assume for a contradiction that this span has an amalgam in the class of involutive residuated lattices. Since $a^2 = 1 > \bot$ while $(cb)^2 = \bot$ and all other elements of $\m{C}$ are idempotent, we may assume that $A = B\cap C$ and that there exists an involutive residuated lattice $\m{D}$ such that $\m{B} \leq \m{D}$ and $\m{C} \leq \m{D}$. Then in $\m{D}$ we have
    \begin{equation}\label{eq2}
     cb \jn acb = (1 \jn a)cb = \top cb = \top b = b,\tag{$\ast\ast$}
    \end{equation}
    where the second to last equality follows from $\top c = \top$.
    Since $cb \leq  c$ and $cb \leq 1$, we get $cb \jn acb \leq c \jn a$, yielding 
    \[
    c \jn a = (cb \jn acb) \jn c \jn a \overset{\eqref{eq2}}{=} b \jn c \jn a = b \jn 1 \jn a = b \jn \top = \top,
    \]
    where the second to last equality follows from $b \jn c = b \jn 1$.
    From $bc = \bot$ and the previous chain of equalities, it follows that
    \[
    ba = bc \jn ba = b(c \jn a) = b\top = \top.
    \]
    Hence
    \[
    b = baa = \top a = \top,
    \]
    a contradiction.
\end{proof}

\begin{figure}
\centering
\begin{tikzpicture}[
fbul/.style={circle,draw=black,fill=black, minimum size = 4pt, inner sep = 0pt},
fcirc/.style={circle,draw=black, minimum size = 4pt, inner sep = 0pt},
square/.style={regular polygon,regular polygon sides=4},
fsqu/.style={square,draw=black,fill=black, minimum size = 5.5pt, inner sep = 0pt}
]
  \node[fbul] (top) at (0,-2.4) {};
  \node[fbul] (1) at (0,-3.2) {};
  \node[fbul] (bot) at (0,-4) {};
   
   \node[right] () at (top) {$\top$};
   \node[right] () at (1) {$1$};
   \node[right] () at (bot) {$\bot$};

   \node[below] (A) at (0,-4.5) {$\m A$};
  
  \draw  (top) -- (1) -- (bot);
\end{tikzpicture}
\quad
\begin{tikzpicture}[
fbul/.style={circle,draw=black,fill=black, minimum size = 4pt, inner sep = 0pt},
fcirc/.style={circle,draw=black, minimum size = 4pt, inner sep = 0pt},
square/.style={regular polygon,regular polygon sides=4},
fsqu/.style={square,draw=black,fill=black, minimum size = 5.5pt, inner sep = 0pt}
]
  \node[fbul] (top) at (0,-3.2) {};
  \node[fcirc] (a) at (0.7,-4) {};
  \node[fbul] (1) at (-0.7,-4) {};
  \node[fbul] (bot) at (0,-4.8) {};
   
   \node[right] () at (top) {$\top = a \top$};
   \node[right] () at (a) {$a$};
   \node[left] () at (1) {$1=a^2$};
   \node[right] () at (bot) {$\bot $};

   \node[below] (B) at (0,-5.3) {$\m B$};

  \draw  (top) -- (a) -- (bot) -- (1) -- (top);
\end{tikzpicture}
\ \ 
\begin{tikzpicture}[
fbul/.style={circle,draw=black,fill=black, minimum size = 4pt, inner sep = 0pt},
fcirc/.style={circle,draw=black, minimum size = 4pt, inner sep = 0pt},
square/.style={regular polygon,regular polygon sides=4},
fsqu/.style={square,draw=black,fill=black, minimum size = 5.5pt, inner sep = 0pt},
squ/.style={square,draw=black,fill=white, minimum size = 5.5pt, inner sep = 0pt}
]
  \node[fsqu] (top) at (0,-2.4) {};
    \node[fbul] (e) at (0,-3.2) {};
  \node[fsqu] (b) at (-0.7,-4) {};
   \node[fsqu] (c) at (0.7,-4) {};
  \node[fbul] (1) at (0,-4) {};
  \node[squ] (d) at (0,-4.8) {};
  \node[fbul] (bot) at (0,-5.6) {};

   \node[right] () at (top) {$\top = b\top = \top c$};
    \node[right] () at (e) {$b\jn c$};
   \node[right] () at (c) {$c = c\top$};
   \node[left] () at (b) {$b = \top b$};
   \node[left] () at (1) {$1$};
   \node[right] () at (d) {$cb$};
   \node[right] () at (bot) {$\bot=bc$};

   \node[below] (C) at (0,-6.1) {$\m C$};

  \draw  (top) -- (e) -- (b) -- (d) -- (c) -- (e) -- (1) -- (d) -- (bot);
\end{tikzpicture}
\caption{
Counterexample for the AP in involutive residuated lattices (Theorem~\ref{thm2}). 
}
\label{fig2}
\end{figure}

\begin{remark}
  The proofs of Theorem~\ref{thm} and Theorem~\ref{thm2} also show that the amalgamation property already fails 
    for the $\{\ld,\rd\}$-free subreducts of residuated lattices, i.e., for lattice-ordered monoids (without the assumption of fusion distributing over meets). Moreover, the proofs do not depend on the meet operation or the constant $1$ being in the signature, so it also follows that the corresponding varieties of  residuated lattice-ordered semigroups,  lattice-ordered semigroups, residuated join-semilattice-ordered semigroups, and join-semilattice-ordered semigroups do not have the amalgamation property.
\end{remark}

The proof of the next theorem follows the strategy used to show that distributive lattices fail the strong amalgamation property, relying on the fact that relative complements are unique in distributive lattices; see, e.g., \cite{Graetzer87}.
\begin{figure}
    \centering
    
\begin{tikzpicture}[
fbul/.style={circle,draw=black,fill=black, minimum size = 4pt, inner sep = 0pt},
fcirc/.style={circle,draw=black, minimum size = 4pt, inner sep = 0pt}
]
  \node[fbul] (1) at (0,3.2) {};
  \node[fcirc] (a) at (0,2.4) {};
  \node[fcirc] (b) at (0,1.6) {};
  \node[fcirc] (c) at (0,0.8) {};
  \node[fbul] (bot) at (0,0) {};
   
   \node[right] () at (1) {$1$};
   \node[right] () at (a) {$a$};
   \node[right] () at (b) {$b$};
   \node[right] () at (c) {$c = ab = a^2$};
   \node[right] () at (bot) {$\bot = b^2 = ac$};

   \node[below] (A) at (0,-0.5) {$\m A$};
   
  \draw  (1) -- (a) -- (b) -- (c) -- (bot);
\end{tikzpicture}
\
\begin{tikzpicture}[
fbul/.style={circle,draw=black,fill=black, minimum size = 4pt, inner sep = 0pt},
fcirc/.style={circle,draw=black, minimum size = 4pt, inner sep = 0pt}
]
  \node[fbul] (1) at (0,3.2) {};
  \node[fcirc] (a) at (0,2.4) {};
  \node[fcirc] (b) at (-0.7,1.6) {};
  \node[fcirc] (p) at (0.7,1.6) {};
  \node[fcirc] (c) at (0,0.8) {};
  \node[fbul] (bot) at (0,0) {};
   
   \node[right] () at (1) {$1$};
   \node[right] () at (a) {$a$};
   \node[left] () at (b) {$b$};
   \node[right] () at (p) {$p$};
   \node[right] () at (c) {$c = bp = ab = a^2$};
   \node[right] () at (bot) {$\bot = p^2 = b^2 = ac$};

   \node[below] (B) at (0,-0.5) {$\m B$};
  
  \draw  (1) -- (a) -- (b) -- (c) -- (bot);
  \draw  (a) -- (p) -- (c);
\end{tikzpicture}
\
\begin{tikzpicture}[
fbul/.style={circle,draw=black,fill=black, minimum size = 4pt, inner sep = 0pt},
fcirc/.style={circle,draw=black, minimum size = 4pt, inner sep = 0pt}
]
  \node[fbul] (1) at (0,3.2) {};
  \node[fcirc] (a) at (0,2.4) {};
  \node[fcirc] (b) at (-0.7,1.6) {};
  \node[fcirc] (q) at (0.7,1.6) {};
  \node[fcirc] (c) at (0,0.8) {};
  \node[fbul] (bot) at (0,0) {};
   
   \node[right] () at (1) {$1$};
   \node[right] () at (a) {$a$};
   \node[left] () at (b) {$b$};
   \node[right] () at (q) {$q$};
   \node[right] () at (c) {$c = q^2 = bq = ab = a^2$};
   \node[right] () at (bot) {$\bot = b^2 = ac$};

   \node[below] (C) at (0,-0.5) {$\m C$};
  
  \draw  (1) -- (a) -- (b) -- (c) -- (bot);
  \draw  (a) -- (q) -- (c);
\end{tikzpicture}
    
    \caption{
    Counterexample for the AP in (integral) distributive residuated lattices
    (Theorem~\ref{thm3}).}
    \label{fig3}
\end{figure}
\begin{theorem}\label{thm3}
    The variety of (commutative) (integral) distributive residuated lattices does not have the amalgamation property.
\end{theorem}
\begin{proof}
    Consider the residuated lattices $\m{A}$, $\m{B}$, and $\m{C}$ depicted in Figure~\ref{fig3}. It is straightforward to check that the inclusion maps $\varphi_B \colon \m{A} \hookrightarrow \m{B}$, and $\varphi_C \colon \m{A}\hookrightarrow \m{C}$ form a span in the variety of distributive residuated lattices. Suppose for a contradiction that the span has an amalgam  in the variety of distributive residuated lattices, i.e., there exists a distributive residuated lattice $\m D$ and embeddings $\psi_B\colon \m{B} \to \m{D}$ and $\psi_C \colon \m{C} \to\m{D}$ such that $\psi_B\circ\varphi_B = \psi_C\circ\varphi_C$. Since $p^2 = \bot$ in $\m{B}$ and $q^2 = c$ in $\m{C}$ we must have $\psi_B(p) \neq \psi_C(q)$, But, since $b \jn p = a$, $b \mt p = c$ and $b\jn q = a$, $b\mt q =c$, both $\psi_B(p)$ and $\psi_C(q)$ are relative complements of $\psi_B(b) = \psi_C(b)$ in the interval $[\psi_B(c), \psi_B(a)] = [\psi_C(c), \psi_C(a)]$. Hence $\psi_B(p) = \psi_C(q)$, a contradiction. In particular, the residuated lattices $\m{A}$, $\m{B}$, and $\m{C}$ are integral and commutative.
\end{proof}

\begin{figure}
\centering

\begin{tikzpicture}[
fbul/.style={circle,draw=black,fill=black, minimum size = 4pt, inner sep = 0pt},
fcirc/.style={circle,draw=black, minimum size = 4pt, inner sep = 0pt},
square/.style={regular polygon,regular polygon sides=4},
fsqu/.style={square,draw=black,fill=black, minimum size = 5.5pt, inner sep = 0pt}
]
  \node[fbul] (top) at (0,-2.4) {};
  \node[fbul] (1) at (0,-3.2) {};
  \node[fbul] (bot) at (0,-4) {};
   
   \node[right] () at (top) {$\top$};
   \node[right] () at (1) {$1$};
   \node[right] () at (bot) {$\bot$};

   \node[below] (A) at (0,-4.5) {$\m A$};

  \draw  (top) -- (1) -- (bot);
\end{tikzpicture}
\qquad
\begin{tikzpicture}[
fbul/.style={circle,draw=black,fill=black, minimum size = 4pt, inner sep = 0pt},
fcirc/.style={circle,draw=black, minimum size = 4pt, inner sep = 0pt},
square/.style={regular polygon,regular polygon sides=4},
fsqu/.style={square,draw=black,fill=black, minimum size = 5.5pt, inner sep = 0pt}
]
   \node[fbul] (top) at (0,-2.4) {};
   \node[fbul] (a) at (0,-3.2) {};
   \node[fbul] (b) at (0.7,-4) {};
   \node[fbul] (1) at (-0.7,-4) {};
   \node[fbul] (bot) at (0,-4.8) {};
   
   \node[right] () at (top) {$\top = \top b$};
   \node[right] () at (a) {$a$};
   \node[right] () at (b) {$b =ab$};
   \node[left] () at (1) {$1$};
   \node[right] () at (bot) {$\bot$};

   \node[below] (B) at (0,-5.3) {$\m B$};
  
  \draw  (top) -- (a) -- (1) -- (bot) -- (b) -- (a);
\end{tikzpicture}
\begin{tikzpicture}[
fbul/.style={circle,draw=black,fill=black, minimum size = 4pt, inner sep = 0pt},
fcirc/.style={circle,draw=black, minimum size = 4pt, inner sep = 0pt},
square/.style={regular polygon,regular polygon sides=4},
fsqu/.style={square,draw=black,fill=black, minimum size = 5.5pt, inner sep = 0pt}
]
  \node[fbul] (top) at (0,-3.2) {};
  \node[fbul] (1) at (-0.7,-4) {};
   \node[fbul] (c) at (0.7,-4) {};
  \node[fbul] (bot) at (0,-4.8) {};
   
   \node[right] () at (top) {$\top$};
   \node[right] () at (c) {$c= \top c$};
   \node[left] () at (1) {$1$};
   \node[right] () at (bot) {$\bot$};
   
   \node[below] (C) at (0,-5.3) {$\m C$};

  \draw  (top) -- (c) -- (bot) -- (1) -- (top);
\end{tikzpicture}

\caption{
Counterexample for the AP in distributive idempotent residuated lattices (Theorem~\ref{thm4}).
}
\label{fig4}
\end{figure}

\begin{theorem}\label{thm4}
    The variety of (commutative) distributive idempotent residuated lattices does not have the amalgamation property.
\end{theorem}

\begin{proof}
    Consider the structures $\m{A}$, $\m{B}$, and $\m{C}$ depicted in Figure~\ref{fig2}. They are commutative distributive idempotent residuated lattices. It is straightforward to check that the inclusion maps $\m{A} \hookrightarrow \m{B}$ and $\m{A} \hookrightarrow \m{C}$ are embeddings, so they form a span in the variety of distributive idempotent  residuated lattices. Assume for a contradiction that this span has an amalgam in the variety of distributive idempotent residuated lattices. Since $c \jn 1 = \top$ and $b\jn 1 = a \jn 1 = a$,
    we may assume that $A = B\cap C$ and that there exists an distributive idempotent residuated lattice $\m{D}$ such that $\m{B} \leq \m{D}$ and $\m{C} \leq \m{D}$. Then, since $c\top = c$ and $\top b = \top$, in $\m{D}$ we have
    \[
    cb = c\top b = c\top = c.
    \]
    So, since $1 \jn c = \top$,
    \[
    \top = \top b = (1\jn c)b = b\jn cb = b \jn c.
    \]
    Now, since $\m D$ is distributive,
    \[
    1  = 1 \mt \top = 1 \mt (b \jn c) = (1 \mt b) \jn (1\mt c) = \bot \jn \bot = \bot,
    \]
    a contradiction.
\end{proof}

Table~\ref{tab:one} summarizes some of the amalgamation results for residuated lattices and marks in boldface which ones follow from our counterexamples. The superscripts denote where the results were proved and a `?' indicates that the problem is still open.

\begin{table}
    \centering
    \caption{Amalgamation in some varieties of residuated lattices} \label{tab:one}
    
    \begin{tabular}{lcc}\toprule
     \qquad \ \
          Variety
     &\qquad\ AP \qquad\ &\!\!\!\!\!\begin{tabular}{c}AP for commutative\\
     subvariety
     \end{tabular}\!\!\!\!\!\\ \midrule 
     \textbf Residuated \textbf Lattices&\textbf{no}\textsuperscript{Thm.~\ref{thm}}&yes\textsuperscript{\cite{GJKO07}} \rule{0pt}{2.6ex} \\
     \textbf Distributive \textbf{RL}&\textbf{no}\textsuperscript{Thm.~\ref{thm3}}&\textbf{no}\textsuperscript{Thm.~\ref{thm3}} \\ 
     \textbf Integral \textbf{RL}&?&yes\textsuperscript{\cite{GJKO07}} \\
     \textbf{Id}empotent \textbf{RL}&\textbf{no}\textsuperscript{Thm.~\ref{thm}}&yes\textsuperscript{\cite{Kamide02}} \\ 
     \textbf{In}volutive \textbf{RL}&\textbf{no}\textsuperscript{Thm.~\ref{thm2}}&yes\textsuperscript{\cite{GJKO07}} \\ 
     \textbf{DIRL}&\textbf{no}\textsuperscript{Thm.~\ref{thm3}}&\textbf{no}\textsuperscript{Thm.~\ref{thm3}}\\
     \textbf{DIdRL}&\textbf{no}\textsuperscript{Thm.~\ref{thm4}}&\textbf{no}\textsuperscript{Thm.~\ref{thm4}} \\ 
     \textbf{DInRL}&no\textsuperscript{\cite{RLR}}&no\textsuperscript{\cite{RLR}}\\
     \bottomrule 
    \end{tabular}
\end{table}

To find the counterexamples in Figure~\ref{fig}--\ref{fig4} we employed a Python program that uses Mace4 \cite{p9} to generate finite spans of residuated lattices and tries to complete them. This way we obtained possible counterexamples, including the ones in Figures~\ref{fig}--\ref{fig4}. 

However, we were not able to resolve the problem for integral residuated lattices. 

\begin{problem}
    Does the variety of integral residuated lattices have the amalgamation property?
\end{problem}

In \cite{Jac2000} Jackson shows that the problem of checking for a span of finite semigroups/rings whether it has a (strong) amalgam in the class of all (finite) semigroups/rings is undecidable. Although this result can be extended to lattice-ordered monoids, it is not clear if the same also holds for residuated lattices.

\begin{problem}
    Is the problem of checking whether a span of finite residuated lattices has an amalgam in the class of (finite) residuated lattices decidable? 
\end{problem}

\subsection*{Acknowledgment.}
The second author was supported by the Swiss National Science Foundation (SNSF), grant no. 200021\textunderscore215157.
\bibliographystyle{plain}

\end{document}